\documentclass{amsart}
\usepackage{a4wide,amssymb}

\newcommand{\U}{\mathcal U}
\newcommand{\V}{\mathcal V}
\newcommand{\C}{\mathcal C}

\newcommand{\vC}{\vec{\mathcal C}}
\newcommand{\w}{\omega}
\newcommand{\IR}{\mathbb R}
\newcommand{\IN}{\mathbb N}
\newcommand{\e}{\varepsilon}

\newcommand{\supp}{\mathrm{supp}}

\newtheorem{theorem}{Theorem}
\newtheorem{example}{Example}

\newtheorem{lemma}{Lemma}
\newtheorem{claim}{Claim}

\theoremstyle{definition}
\newtheorem{remark}{Remark}

\title[The strong universality of ANRs with a suitable algebraic structure]{The strong universality of ANRs with\\ a suitable algebraic structure}
\author{Taras Banakh}
\address{Ivan Franko National University of Lviv (Ukraine) and Jan Kochanowski University in Kielce (Poland)}
\email{t.o.banakh@gmail.com}

\begin{document}
\begin{abstract} Let $M$ be an ANR space and $X$ be a homotopy dense subspace in $M$. Assume that $M$ admits a continuous binary operation $*:M\times M\to M$ such that for every $x,y\in M$ the inclusion $x*y\in X$ holds if and only if $x,y\in X$. Assume also that there exist continuous unary operations $u,v:M\to M$ such that $x=u(x)*v(x)$ for all $x\in M$. Given a $2^\w$-stable $\mathbf \Pi^0_2$-hereditary weakly $\mathbf \Sigma^0_2$-additive class of spaces $\C$, we prove that the pair $(M,X)$ is strongly $(\mathbf \Pi^0_1\cap\C,\C)$-universal if and only if for any compact space $K\in\C$, subspace $C\in\C$ of $K$ and nonempty open set $U\subseteq M$ there exists a continuous map $f:K\to U$ such that $f^{-1}[X]=C$. This characterization is applied to detecting strongly universal Lawson semilattices.
\end{abstract}
\maketitle

\section{Introduction and Main Results}

The strong universality is one of ingredients in characterizations of infinite-dimensional manifolds. It can be defined for spaces, pairs and more generally, systems of spaces \cite{BGM}. All topological spaces in this paper are metrizable and separable, and all maps are continuous.

We start with recalling the definition of strong universality for spaces and then will turn to pairs of spaces.

Let $Y$ be a topological space and $\U$ be an open cover of $Y$. Two maps $f,g:X\to Y$ on a topological space $X$ are called {\em $\U$-near} if for any $x\in X$ there exists a set $U\in\U$ such that $\{f(x),g(x)\}\subseteq U$. 

A subset $A$ of a topological space $X$ is called {\em a $Z$-set} in $X$ if $A$ is closed in $X$ and for every open cover $\U$ of $X$ there exists a continuous function $f:X\to X\setminus A$ which is $\U$-near to the identity function of $X$.

A function $f:X\to Y$ between topological spaces is called a {\em $Z$-embedding} if $f$ is a topological embedding whose image $f[X]$ is a $Z$-set in $Y$.


Let $C$ be a space. A topological space $X$ is called {\em strongly $C$-universal} if for every open cover $\U$ of $X$, closed subset $B\subseteq C$ and map $f:C\to X$ such that $f{\restriction}B:B\to X$ is a $Z$-embedding there exists a $Z$-embedding $g:C\to X$ such that $g{\restriction}B=f{\restriction}B$ and $g$ is $\U$-near to $f$.

Let $\C$ be a class of spaces. A topological space $X$ is called
\begin{itemize}
\item {\em $\C$-universal} if for every $C\in\C$ there exists a closed topological embedding $f:C\to X$;
\item {\em everywhere $\C$-universal} if for every nonempty open set $U\subseteq X$ there exists a closed topological embedding $f:C\to X$ such that $f[C]\subseteq U$;
\item  {\em strongly $\C$-universal} if $X$ is strongly $C$-universal for every space $C\in\C$.
\end{itemize}

By $\mathbf\Pi^0_1$ and $\mathbf \Pi^0_2$ we denote the classes of compact metrizable and Polish spaces, respectively.

\begin{example}
\begin{enumerate}
\item The Hilbert cube $[0,1]^\w$ is strongly $\mathbf \Pi^0_1$-universal \cite{Tor80}.
\item The countable product of lines $\IR^\w$ is strongly $\mathbf \Pi^0_2$-universal  \cite{Tor81}.
\end{enumerate}
\end{example}

A subset $A$ of a topological space $X$ is called a 
\begin{itemize}
\item a {\em retract} in $X$ if there exists a continuous map $r:X\to A$ such that $r(a)=a$ for all $x\in A$;
\item a {\em neighborhood retract} if $A$ is closed in $X$ and $A$ is a retract of some open set $U\subseteq X$ containing $A$.
\end{itemize}

A topological space $X$ is called an {\em absolute} ({\em neighborhood}) {\em retract} if $X$ is metrizable and $X$ is a (neighborhood) retract in each metrizable space containing $X$ as a closed subspace. Absolute neighborhood retracts (briefly, ANRs) play an important role in Geometric and Infinite-Dimensional Topology \cite{Sak1}, \cite{Sak2}.

A topological space $X$ is defined to satisfy the {\em strong discrete approximation property} (briefly, SDAP) if for every open cover $\U$ of $X$ and any map $f:\oplus_{n\in\IN}[0,1]^n\to X$ from the topological sum of finite-dimensional cubes,  there exist a map $g:\oplus_{n\in\IN}[0,1]^n\to X$ such that $g$ is $\U$-near to $f$ and the family $\big(g\big[[0,1]^n\big]\big)_{n\in\IN}$ is discrete in $X$, which means that each point $x\in X$ has a neighborhood $O_x\subseteq X$ that meets at most one set $g\big[[0,1]^n\big]$, $n\in\IN$.

The strong discrete approximation property is a crucial ingredient in the famous Toru\'nczyk characterization of $\IR^\w$-manifolds \cite{Tor80}. A topological space $M$ is called a {\em manifold modeled on a space $E$} (briefly, an {\em $E$-manifold}) if each point $x\in M$ has a neighborhood, homeomorphic to an open subset of the space $E$.

\begin{theorem}[Toru\'nczyk] A topological space $X$ is an $\IR^\w$-manifold if and only if $X$ is a Polish ANR satisfying SDAP.
\end{theorem}

A subset $X$ of a topological space $M$ is called {\em homotopy dense} if there exists a continuous map $H:M\times[0,1]\to M$ such that $H(x,0)=x$ and $H(x,t)\in X$ for all $x\in M$ and $t\in(0,1]$. 

The following characterization of ANRs with SDAP was proved by the author in \cite{Ban98}.

\begin{theorem}[Banakh] A topological space is an ANR with SDAP if and only if $X$ is homeomorphic to a homotopy dense subset of an $\IR^\w$-manifold.
\end{theorem}

Let $\C$ be a class of topological spaces. A topological space $X$ is called {\em $\C$-absorbing} if 
\begin{itemize}
\item $X$ is an absolute neighborhood retract satisfying SDAP;
\item $X=\bigcup_{n\in\w}X_n$ where each $X_n$ is a $Z$-set in $X$ and $X_n\in\C$;
\item $X$ is strongly $\C$-universal.
\end{itemize}

A class $\C$ of spaces is called {\em topological} if for any space $X\in\C$ and any homeomorphism $h:X\to Y$ the space $Y$ also belongs to the class $\C$. 

The following characterization theorem due to Bestvina and Mogilski \cite{BM} is one of the most important results of infinite-dimensional topology. Its proof can be found in \cite[1.6.3]{BRZ}.

\begin{theorem}[Bestvina--Mogilski]\label{t:BM} Let $\C$ be a class of topological spaces. Two $\C$-absorbing spaces are homeomorphic if and only if they are homotopically equivalent. In particular, any two $\C$-absorbing absolute retracts are homeomorphic.
\end{theorem}

A similar characterization theorem holds for pairs of spaces. By a {\em pair} of topological spaces we understand an ordered pair $(X,Y)$ of topological spaces $Y\subseteq X$. A pair of spaces $(K,X)$ is called {\em compact} if the space $K$ is compact. Two pairs $(X,Y)$ and $(X',Y')$ are {\em homeomorphic} if there exists a homeomorphism $h:X\to X'$ such that $h[Y]=Y'$.

For two classes of spaces $\C,\mathcal D$ denote by $(\C,\mathcal D)$ the class of pairs $(C,D)$ such that $C\in\C$, $D\in\mathcal D$, and $D\subseteq C$. A class of pairs $\vC$ is called {\em topological} if for any pair $(X,C)\in\vC$ and any homeomorphism $h:X\to Y$ the pair $(Y,h[C])$ belongs to the class $\vC$.

Let $(K,C)$ be a pair of spaces. A pair of spaces $(X,Y)$ is called {\em strongly $(K,C)$-universal} if for every open cover $\U$ of $X$, closed subset $B\subseteq K$ and map $f:K\to X$ such that $f{\restriction}B$ is a $Z$-embedding with $(f{\restriction}B)^{-1}[Y]=B\cap C$, there exists a $Z$-embedding $g:K\to X$ such that $g{\restriction}B=f{\restriction}B$, $g^{-1}[Y]=C$, and $g$ is $\U$-near to $f$.

Let $\vC$ be a class of pairs. A pair of topological spaces $(M,X)$ is called
\begin{itemize}
\item {\em $\vC$-universal} if for every pair $(K,C)\in \vC$ there exists a closed topological embedding $f:K\to M$ such that $f^{-1}[X]=C$;
\item {\em everywhere $\vC$-universal} if for every nonempty open set $U\subseteq M$ and pair $(K,C)\in \vC$ there exists a closed topological embedding $f:K\to M$ such that $f^{-1}[X]=C$ and $f[K]\subseteq U$;
\item {\em strongly $\vC$-universal} if it is strongly $(K,C)$-universal for every pair $(K,C)\in\vC$. 
\end{itemize}

A pair of spaces $(X,Y)$ is called {\em $\vec \C$-absorbing} if 
\begin{itemize}
\item the pair $(X,Y)$ is strongly $\vC$-universal;
\item there exists a sequence $(X_n)_{n\in\w}$ of $Z$-sets in $X$ such that $Y\subseteq \bigcup_{n\in\w}X_n$ and $(X_n,X_n\cap Y)\in\vec \C$ for every $n\in\w$.
\end{itemize}

The following pair counterpart of Bestvina--Mogilski Theorem~\ref{t:BM} can be found in \cite[1.7.7]{BRZ}.

\begin{theorem} Let $\vec\C$ be a topological class of pairs and $M$ be a manifold modeled on $[0,1]^\w$ or $\IR^\w$. Any two $\vec\C$-absorbing pairs $(M,X)$ and $(M,Y)$ are homeomorphic.
\end{theorem}

In some cases, the strong $\C$-universality can be reduced to the strong $(\mathbf\Pi^0_1\cap\C,\C)$-universality, where $\mathbf \Pi^0_1$ denotes the class of compact metrizable spaces. The following proposition was proved in \cite[3.1]{BC}.

\begin{theorem}[Banakh--Cauty]\label{p1} Let $\C$ be a class of spaces, $M$ be an ANR and $X$ be a homotopy dense subset of $M$ such that $X$ has SDAP. If for some pair $(K,C)$,  the pair $(M,X)$ is strongly $(K,C)$-universal, then the space $X$ is strongly $C$-universal.
\end{theorem}

For some nice classes $\C$ the strong $\C$-universality implies the strong $(\mathbf \Pi^0_1\cap\C,\C)$-universality.

A class $\C$ of spaces is defined to be
\begin{itemize}
\item {\em $2^\w$-stable} if for any space $C\in\C$ the product $C\times 2^\w$ of $C$ with the Cantor cube $2^\w=\{0,1\}^\w$ belongs to the class $\C$;
\item {\em compactification admitting} if each space $C\in\C$ is contained in some compact space in the class $\C$;
\item {\em $\mathbf \Pi^0_2$-hereditary} if for any space $C\in\C$, any $G_\delta$-subset of $C$ belongs to the class $\C$;
\item ({\em weakly}) {\em $\mathbf \Sigma^0_1$-additive} if for any compact metrizable space $K$ (with $K\in\C)$, subset $C\subseteq K$ with $C\in\C$ and  $\sigma$-compact set $A\subseteq K$, the union $C\cup A$ belongs to the class $\C$.
\end{itemize}

The following theorem was proved in \cite[4.1]{BC}.

\begin{theorem}[Banakh--Cauty]\label{t:BC} Let $M$ be a Polish ANR and $X$ be a homotopy dense subset in $M$. If the space $X$ is strongly $\C$-universal for some $2^\w$-stable weakly $\mathbf \Sigma^0_1$-additive class of spaces $\C$, then the pair $(M,X)$ is strongly $(\mathbf\Pi^0_1\cap\C,\C)$-universal.
\end{theorem}

For some nice classes of pairs $\vC$ the strong $\vC$-universality is equivalent to the strong $\vC$-preuni\-versality. 

Let $\vC$ be a classes of compact pairs. A pair $(M,X)$ of topological spaces is called
\begin{itemize}
\item {\em $\vC$-preuniversal} if for every pair $(K,C)\in\vC$ there exists a  map $f:K\to M$ such that $f^{-1}[X]=C$;
\item {\em everywhere $\vC$-preuniversal} if for every nonempty open set $U\subseteq M$ and pair $(K,C)\in\vC$ there exists a  map $f:K\to M$ such that $f^{-1}[X]=C$ and $f[K]\subseteq U$;
\item {\em strongly $\vC$-preuniversal} if for every open cover $\U$ of $X$, pair $(K,C)\in\vC$, closed set $B\subseteq K$ and map $f:K\to M$ such that $(f{\restriction}B)^{-1}[X]=B\cap C$, there exists a map $g:K\to M$ such that $g{\restriction}B=f{\restriction}B$, $g^{-1}[X]=C$ and $g$ is $\U$-near to $f$.
\end{itemize}

The following useful characterization was proved in \cite[\S3.2]{BRZ}.

\begin{theorem}\label{t5} Let $\C$ be a $2^\w$-stable $\mathbf \Pi^0_2$-hereditary weakly $\mathbf \Sigma^0_2$-additive class of spaces. Let $(M,X)$ be a pair of spaces (such that $M$ is an ANR and $X$ is a homotopy dense subset in $M$). The pair $(M,X)$ is (strongly) $(\mathbf \Pi^0_1\cap\C,\C)$-universal if and only if it is (strongly) $(\mathbf \Pi^0_1\cap\C,\C)$-preuniversal.
\end{theorem}

We have eventually arrived to the main topic of this paper: the strong universality of spaces carrying some algebraic structure. For the structures of a topological group or a convex set this topic has been elaborated in \cite[\S4.2, \S5.3]{BRZ}, \cite{Ban} and \cite{BDP}. 

In this paper we study the strong universality in some topological magmas.

Following Bourbaki, by a {\em magma} we understand a set $M$ endowed with a binary operation $*:M\times M\to M$. 

A magma $M$ is called 
\begin{itemize}
\item {\em commutative} if $x*y=y*x$ for all $x,y\in M$;
\item {\em unital} if  there exists an element $e\in M$ such that $e*x=x=x*e$ for all $x\in M$;
\item {\em idempotent} if $x*x=x$ for all $x\in M$;
\item a {\em semigroup} if its binary operation is associative;
\item a {\em monoid} if $M$ is a unital semigroup;
\item a {\em band} if $M$ is an idempotent semigroup;
\item a {\em semilattice} if $M$ is a commutative band.
\end{itemize}

For subsets $A,B$ of a magma, we put $A*B=\{a*b:a\in A,\;b\in B\}$. 

A subset $A$ of a topological magma $X$ is called
\begin{itemize}
\item a {\em submagma} if $A*A\subseteq A$;
\item an {\em ideal} if $(A*X)\cup (X*A)\subseteq A$;
\item a {\em coideal} if $A$ is a submagma and $X\setminus A$ is an ideal in $X$.
\end{itemize}

A {\em topological magma} is a topological space $M$ endowed with a continuous binary operation $*:M\times M\to M$. 

A topological magma $X$ is called {\em decomposable} if there exist continuous unary operations $u,v:X\to X$ such that $x=u(x)*v(x)$ for every $x\in M$. 

Observe that a topological magma is decomposable if it is unital or idempotent.

The following general theorem is the main result of this paper.

\begin{theorem}\label{t:main} Let $X$ be a coideal in a decomposable topological magma $M$. If $M$ is an ANR and $X$ is homotopy dense in $M$, then for any class of compact pairs $\vC$, the pair $(M,X)$ is strongly $\vC$-preuniversal if and only if it is everywhere $\vC$-preuniversal.
\end{theorem}

The proof of Theorem~\ref{t:main} is long and technical. It will be presented in Section~\ref{s:main}. Now we will apply Theorem~\ref{t:main} to detecting strongly universal topological semilattices.

A {\em topological semilattice} is a topological space $X$ endowed with a commutative associative operation $*:X\times X\to X$ such that $x*x=x$ for all $x\in X$. A topological semilattice $X$ is called a {\em Lawson semilattice} if it has a base of the topology consisting of subsemilattices. Lawson semilattices often appear in the theory of hyperspaces, see \cite{KSY}, \cite{DR}, \cite{GM}, \cite{Krup}, \cite{KK}, \cite{KS07}, \cite{KS}.

A subset $A$ of a topological space $X$ is called {\em locally path connected in $X$} if for every point $x\in X$ and neighborhood $O_x\subseteq X$ of $x$ there exists a neighborhood $U_x\subseteq X$ of $x$ such that for any points $y,z\in A\cap U_x$ there exists a continuous map $\gamma:[0,1]\to A\cap O_x$ such that $\gamma(0)=y$ and $\gamma(1)=z$. A topological space $X$ is called {\em locally path connected} if $X$ is locally path connected in $X$.

The following helpful result was proved by Kubi\'s, Sakai and Yaguchi in \cite[5.1]{KSY}.

\begin{theorem}[Kubi\'s--Sakai--Yaguchi]\label{t:KSY} Let $M$ be a Lawson semilattice and $X$ be a dense locally path connected subsemilattice in $X$. Then $M$ is an ANR and $X$ is homotopy dense in $M$.
\end{theorem}

Applying Theorems~\ref{t:main} and \ref{t:KSY} to topological semilattices, we obtain the following useful characterization.

\begin{theorem}\label{t:Lawson} Let $M$ be a Lawson semilattice and $X$ be a dense coideal in $M$ such that $X$ is locally path connected in $M$. For a $2^\w$-stable $\mathbf \Pi^0_2$-hereditary weakly $\mathbf\Sigma^0_2$-additive class of spaces $\C$, the following conditions are equivalent:
\begin{enumerate}
\item the pair $(M,X)$ is strongly $(\mathbf \Pi^0_1\cap\C,\C)$-universal;
\item $(M,X)$ is strongly $(\mathbf \Pi^0_1\cap\C,\C)$-preuniversal;
\item $(M,X)$ is everywhere $(\mathbf \Pi^0_1\cap\C,\C)$-preuniversal.
\end{enumerate}
If the class $\C$ is compactification admitting and the space $X$ has SDAP, then the conditions \textup{(1)--(3)} are equivalent to:
\begin{enumerate}
\item[(4)] the space $X$ is strongly $\C$-universal.
\end{enumerate}
\end{theorem} 

\begin{proof} Since the subsemilattice $X$ is locally path connected in $M$, we can apply Theorem~\ref{t:KSY} and conclude that the Lawson semilattice $M$ is an ANR and its dense subsemilattice $X$ is homotopy dense in $M$. Applying Theorem~\ref{t:main} (with $u,v$ being identity maps of $M$), we conclude that the conditions (2) and (3) are equivalent. The conditions (1) and (2) are equivalent by Theorem~\ref{t5}. If the class $\C$ is compactification admitting and the space $X$ has SDAP, then the equivalence $(1)\Leftrightarrow(4)$ follows from Theorems~\ref{p1} and \ref{t:BC}.
\end{proof}

\begin{remark} Theorem~\ref{t:Lawson} has been applied in \cite{BKO} to studying the topological structure of some hyperspaces. For more applications of strongly universal and absorbing spaces in the theory of hyperspaces, see \cite{BanMaz}, \cite{Cauty}, \cite{DR}, \cite{GM}, \cite{Krup}, \cite{KK}, \cite{KS07}, \cite{KS}.
\end{remark}

\section{Proof of Theorem~\ref{t:main}}\label{s:main}

Let $M$ be an ANR space carrying a continuous binary operation $*:M\times M\to M$ and two continuous unary operations $u,v:M\to M$ such that $u(x)*v(x)=x$ for all $x\in M$. Let $X$ be a homotopy dense subset of $M$ such that $X*X\subseteq X$ and $(M*(M\setminus X))\cup((M\setminus X)*M)\subseteq M\setminus X$. Two latter inclusions are equivalent to the equivalence
$$x*y\in X\;\Leftrightarrow\;x,y\in X$$holding for any elements $x,y\in M$.

Given any class $\vC$ of compact pairs, we should prove that the pair $(M,X)$ is strongly $\vC$-preuniversal if it is everywhere $\vC$-preuniversal (the opposite implication being obvious). So, we assume that the pair $(M,X)$ is everywhere $\vC$-preuniversal.

Fix a metric $d$ generating the topology of the space $M$. For any $x\in M$ and positive real number $\e$ denote by $B(x,\e)$ the open $\e$-ball $\{y\in M:d(y,x)<\e\}$ centered at $x$.

\begin{lemma}\label{l1} For every open cover $\U$ of $M$ and every compact subset $K\subseteq M$ there exists $\e>0$ such that for every $x\in K$ there exists a set $U\in\U$ such that $B(x,\e)\subseteq U$.
\end{lemma}

\begin{proof} For every $x\in K$ we can find an open set $U_x\in\U$ containing $x$ and a positive real number $\e_x$ such that $B(x,2\e_x)\subseteq U_x$. By the compactness of $K$, there exists a finite set $C\subseteq K$ such that $K\subseteq\bigcup_{c\in C}B(c,\e_c)$. We claim that the number $\e=\min\{\e_c:c\in C\}>0$ has the required property. Indeed, for any $x\in K$ we can find $c\in C$ such that $x\in B(c,\e_c)$. Then $B(x,\e)\subseteq B(c,\e_c+\e)\subseteq B(c,2\e_c)\subseteq U_c\in\U$.
\end{proof}

\begin{lemma}\label{l:lambda1} For any open set $W\subseteq M\times [0,1]$ containing the set $M\times\{0\}$, there exists a continuous function $\lambda:M\to(0,1]$ such that
$$\{(x,t)\in M\times[0,1]:0\le t\le\lambda(x)\}\subseteq W.$$
\end{lemma}

\begin{proof} For every $x\in M$, find an open neighborhood $W_x\subseteq M$ of $x$ and a positive real number $\e_x$ such that $W_x\times[0,\e_x]\subseteq W$. By the paracompactness of the metrizable space $M$, there exists a locally finite open cover $\U$ of $M$ such that each set $U\in\U$ is contained in some set $W_x$. So, for some $\e_U>0$ (for example, $\e_U=\e_x$), we have $U\times [0,\e_U]\subseteq W$. By the paracompactness of $M$ there exists a family of continuous functions $(\lambda_U:M\to[0,1])_{U\in\U}$ such that 
\begin{itemize}
\item $\supp(\lambda_U)=\lambda_U^{-1}\big((0,1]\big)\subseteq U$ for every $U\in\U$ and
\item $\sum_{U\in\U}\lambda_U(x)=1$ for every $x\in M$.
\end{itemize}
Finally, consider the function $$\lambda:M\to(0,1],\quad\lambda:x\mapsto\sum_{U\in\U}\lambda_U(x)\cdot\e_U,$$and observe that $\lambda$ is well-defined, continuous,  and has the required property:
$$\{(x,t)\in M\times[0,1]:t\le \lambda(x)\}\subseteq W.$$
\end{proof}

\begin{lemma}\label{l:Hom} There exists a homotopy $H:M\times[0,1]\to M$ such that for every $x\in M$ and $t\in(0,1]$ the following conditions are satisfied:
\begin{enumerate}
\item $H(x,0)=x$;
\item $H(x,t)\in X$;
\item $d(x,H(x,t))< t$.
\end{enumerate}
\end{lemma}

\begin{proof} Since $X$ is homotopy dense in $M$, there exists a homotopy $h:M\times [0,1]\to M$ such that $h(x,0)=x$ and $h(x,t)\in X$ for all $x\in M$ and $t\in(0,1]$. For every $n\in\w$ consider the open subset $$W_n=\{(x,t)\in M\times[0,1]:h(x,t) \in B(x,2^{-n})\}$$of the product $M\times[0,1]$, and observe that $M\times\{0\}\subseteq W_n$. Using Lemma~\ref{l:lambda1}, we can construct a sequence of continuous functions $(\lambda_n:M\to(0,1])_{n\in\w}$ such that for every $n\in\w$ the following conditions are satisfied:
\begin{itemize}
\item $\{(x,t)\in M\times[0,1]:t\le\lambda_n(x)\}\subseteq W_{n+1}\cap(M\times[0,2^{-n}])$;
\item if $n>0$, then $\lambda_n(x)<\lambda_{n-1}(x)$ for all $x\in M$.
\end{itemize}

Let $\xi:M\times[0,1]\to [0,1]$ be the continuous map defined by the formula
$$\xi(x,t)=\begin{cases}
\lambda_{n}(x)+\big(\lambda_{n-1}(x)-\lambda_{n}(x)\big)(2^{n}t-1)&\mbox{if $\frac1{2^{n}}\le t\le \frac1{2^{n-1}}$ for some $n\in\IN$};\\
0&\mbox{otherwise}.
\end{cases}
$$
We claim that the continuous map
$$H:M\times[0,1]\to M,\quad H:(x,t)\mapsto h(x,\xi(x,t)),$$
has the required properties. Fix any $x\in M$ and $t\in(0,1]$. Then $\xi(x,t)\in(0,1]$ and hence  $H(x,t)=h(x,\xi(x,t))\in X$ and $H(x,0)=h(x,\xi(x,0))=h(x,0)=x$.  Find a unique $n\in\IN$ such that $\frac1{2^n}<t\le\frac1{2^{n-1}}$. The definition of the map $\xi$ ensures that $\{x\}\times \xi(x,t)\in \{x\}\times [\lambda_n(x),\lambda_{n-1}(x)]\subseteq W_n$ and hence $$d(x,H(x,t))=d(x,h(x,\xi(x,t)))<2^{-n}<t.$$
\end{proof}

From now on, we fix a homotopy $H:M\times [0,1]\to M$ satisfying the conditions (1)--(3) of Lemma~\ref{l:Hom}.

\begin{lemma}\label{l:lambda} There exist an open neighborhood $W_\lambda\subseteq M\times M$ of the diagonal and a continuous map $\lambda:W_\lambda\times [0,1]\to M$ such that for every $(x,y)\in W_\lambda$, the following conditions are satisfied:
\begin{enumerate}
\item $\lambda(x,y,0)=x$ and $\lambda(x,y,1)=y$;
\item $\lambda(x,x,t)=x$ for every $t\in[0,1]$;
\item if $x\ne y$, then $\lambda(x,y,t)\in X$ for every $t\in(0,1)$.
\end{enumerate} 
\end{lemma}

\begin{proof}  By Arens--Eells Embedding Theorem 1.13.2 \cite{Sak2}, the metric space $(M,d)$ can be identified with a closed subspace of some normed space $L$. Since $X$ is an ANR, there exists a retraction $r:W\to X$ of some open neighborhood $W$ of $X$ in $L$. Consider the open set $$W_\lambda=\{(x,y)\in M\times M:[x,y]\subseteq W\}$$where $[x,y]=\{tx+(1-t)y:t\in[0,1]\}$ is the closed segment connecting the points $x,y\in M\subseteq L$ in the linear space $L$. 

We claim that the continuous map $$\lambda:W_\lambda\times [0,1]\to M,\quad\lambda:((x,y),t)\mapsto H(r((1-t)x+ty),\min\{t,1-t,d(x,y)\}),$$has the required properties.
Indeed, for any $(x,y)\in W_\lambda$ and $t\in [0,1]$  we have 
$$\begin{aligned}
&\lambda(x,y,0)=H(r(x),\min\{0,1,d(x,y)\})=H(r(x),0)=r(x)=x;\\
&\lambda(x,y,1)=H(r(y),\min\{1,0,d(x,y)\})=H(r(y),0)=r(y)=y;\\
&\lambda(x,x,t)=H(r((1-t)x+tx),\min\{t,1-t,d(x,x)\})=H(r(x),0)=r(x)=x.
\end{aligned}
$$
If $x\ne y$ and $t\in(0,1)$, then $\min\{t,1-t,d(x,y)\}>0$ and hence
$$\lambda(x,y,t)=H(r((1-t)x+ty),\min\{t,1-t,d(x,y)\})\in X$$by Lemma~\ref{l:Hom}(2).\end{proof}

From now on, we fix an open neighborhood $W_\lambda\subseteq M\times M$ of the diagonal and a continuous map $\lambda:W_\lambda\times[0,1]\to M$ satisfying the conditions (1)--(3) of Lemma~\ref{l:lambda}.

\begin{lemma}\label{l:lambda2} For every compact set $K\subseteq M$ and any $\e>0$ there exists $\delta>0$ such that for every $x\in K$ we have
$$B(x,\delta)\times B(x,\delta)\subseteq W_\lambda\quad\mbox{and}\quad \lambda\big[B(x,\delta)\times B(x,\delta)\times[0,1]\big]\subseteq B(x,\tfrac12\e).$$
\end{lemma}

\begin{proof} For every $x\in K$ we have $\lambda\big[\{x\}\times\{x\}\times[0,1]\big]=\{x\}\subseteq B(x,\frac12\e)$. By the compactness of the interval $[0,1]$ and the continuity of $\lambda$, there exists a positive $\delta_x<\frac12\e$ such that $$\lambda\big[B(x,2\delta_x)\times B(x,2\delta_x)\times[0,1]\big]\subseteq B(x,\tfrac12\e).$$ By the compactness of $K$, there exists a finite set $C\subseteq K$ such that $K\subseteq\bigcup_{c\in C}B(c,\delta_c)$. We claim that the number $\delta=\min\{\delta_c:c\in C\}>0$ has the required property. Indeed, for any $x\in K$ we can find $c\in C$ such that $x\in B(c,\delta_c)$. Then for every $y,z\in B(x,\delta)$ we have $y,z\in B(x,\delta)\subseteq B(B(c,\delta_c),\delta)\subseteq B(c,\delta_c+\delta)\subseteq B(c,2\delta_c)$ and hence $$\lambda(y,z,t)\in \lambda\big[B(c,2\delta_c)\times B(c,2\delta_c)\times[0,1]\big]\subseteq B(c,\tfrac12\e)\subseteq B(x,\delta_c+\tfrac12\e)\subseteq B(x,\e).$$ 
\end{proof}

We recall that the space $M$ carries a continuous binary operation $*:M\times M\to M$ and two continuous unary operations $u,v:M\to M$ such that $u(x)*v(x)=x$ for all $x\in M$. By analogy with Lemma~\ref{l:lambda2} the following lemma can be proved.

\begin{lemma} For every compact set $K\subseteq M$ and every $\e>0$, there exists $\delta>0$ such that for every $x\in K$ $$B(u(x),\delta)*B(v(x),\delta)\subseteq B(x,\e).$$
\end{lemma}

\begin{lemma}\label{l:operations} For every $n\in\IN$ there exist a continuous $n$-ary operation $*_n:M^n\to M$ and continuous unary operations $w_1,\dots,w_n:M\to M$ such that
\begin{enumerate}
\item $x=*_n(w_1(x),\dots,w_n(x))$ for every $x\in M$;
\item for every $x_1,\dots,x_n\in M$ we have $*_n(x_1,\dots,x_n)\in X$ if and only if $x_1,\dots,x_n\in X$.
\end{enumerate}
\end{lemma}

\begin{proof} Let us recall that $M$ carries a continuous binary operation $*:M\times M\to M$ and continuous unary operations $u,v:M\to M$ such that $x=u(x)*v(x)$ for all $x\in M$ and for every $x_1,x_2\in M$ we have $x_1*x_2\in X$ iff $x_1,x_2\in X$. For every $n\in\IN$, define a continuous $n$-ary operation $*_n:M^n\to M$ by the recursive formula $*_1(x_1)=x_1$ and $$*_{n+1}(x_1,\dots,x_n,x_{n+1})=*_n(x_1,\dots,x_n)*x_{n+1}$$where $x_1,\dots,x_{n+1}\in M$. 

\begin{claim} For every $n\in\IN$ and $x_1,\dots,x_{n}\in M$ we have $*_n(x_1,\dots,x_n)\in X$ if and only if $x_1,\dots,x_n\in X$.
\end{claim} 

\begin{proof} For $n=1$ we have $*_n(x_1)=x_1\in X$ if and only if $x_1\in X$. Assume that for some $n\in\IN$ we have proved that $*_n(x_1,\dots,x_n)\in X$ iff $x_1,\dots,x_n\in X$. 

Then for every $x_1,\dots,x_{n+1}\in M$ we have
$*_{n+1}(x_1,\dots,x_{n+1})=*_n(x_1,\dots,x_n)*x_{n+1}\in X$ if and only if $*_n(x_1,\dots,x_n),x_{n+1}\in X$ if and only if $x_1,\dots,x_n,x_{n+1}\in X$.
\end{proof}

\begin{claim} For every $n\in\IN$ there exist continuous unary operations $w_1,\dots,w_n:M\to M$ such that $$*_n(w_1(x),\dots,w_n(x))=x$$for every $x\in M$.
\end{claim}

\begin{proof} For $n=1$ let $w_1$ be the identity map of $M$. In this case  $*_1(w_1(x))=w_1(x)=x$ for all $x\in M$. Assume that for some $n\in\IN$ we have found continuous unary operations $w_1,\dots,w_n:M\to M$ such that $*_n(w_1(x),\dots,w_n(x))=x$ for all $x\in M$. Consider the unary operations $w_1',\dots,w_{n+1}':M\to M$ defined by the formulas $w'_k=w_k\circ u$ for $k\le n$ and $w'_{n+1}=v$. Then for every $x\in M$ we have
$$
\begin{aligned}
*_{n+1}(w_1'(x),\dots,w'_{n+1}(x))&=*_n(w_1'(x),\dots,w'_n(x))*w'_{n+1}(x)=\\
&=*_n\big(w_1(u(x)),\dots,w_n(u(x))\big)*v(x)=u(x)*v(x)=x.
\end{aligned}
$$
\end{proof} 
\end{proof}

We shall need the following characterization of dimension due to Ostrand \cite{Ost}, see also \cite[3.2.4]{End}.

\begin{lemma}[Ostrand]\label{l:Ostrand} A space $Z$ has covering dimension $\dim(Z)<n$ for some $n\in\IN$ if and only if for any open cover $\U$ of $Z$ there exists an open cover $\V$ of $X$ such that
\begin{enumerate}
\item each set $V\in\V$ is contained in some set $U\in\U$;
\item $\V=\bigcup_{i=1}^n\V_i$ where each family $\V_i$ consists of pairwise disjoint sets. 
\end{enumerate}
\end{lemma}

\begin{lemma}\label{l:step} For any $\e>0$, pair $(K,C)\in\vC$, closed subspace $F\subseteq K$, finite dimensional compact space $Z$ and  continuous maps $f:F\to Z$, $g:Z\to X$, there exists a continuous map $h:F\to M$ such that $d(h,g\circ f)<\e$ and $h^{-1}[X]=F\cap C$.
\end{lemma}

\begin{proof} Let $m=1+\dim(Z)$. 
By Lemma~\ref{l:operations}, there exist a continuous $m$-ary operation $*_m:M^m\to M$ and continuous unary operations $w_1,\dots,w_m:M\to M$ such that 
\begin{itemize}
\item[(i)] $x=*_m(w_1(x),\dots,w_m(x))$ for all $x\in M$;
\item[(ii)] for any $x_1,\dots,x_m\in M$ the inclusion $*_m(x_1,\dots,x_m)\in X$ holds iff $x_1,\dots,x_m\in X$.
\end{itemize}

Using the compactness of the set $g[Z]\subseteq M$ and the identity (i), we can find $\delta>0$ such that for every $x\in g[Z]$ and every sequence $(x_i)_{i=1}^m\in \prod_{i=1}^m B(w_i(x),2\delta)$ we have $*_m(x_1,\dots,x_m)\in B(x,\e)$.

 By Lemma~\ref{l:lambda2}, there exists a positive $\eta<\delta$ such that for every $x\in \bigcup_{i=1}^mw_i\circ g[Z]$ we have $$B(x,\eta)\times B(x,\eta)\subseteq W_\lambda\quad\mbox{and}\quad\lambda\big[B(x,\eta)\times B(x,\eta)\times[0,1]\big]\subseteq  B(x,\delta).$$ By the continuity of $g$ and compactness of $Z$, there exists a finite cover $\U$ of $Z$ by nonempty open sets such that for each set $U\in\U$ and each $k\in\{1,\dots,m\}$ the image $w_k\circ g[U]$ has diameter $<\eta$ in the metric space $(M,d)$. By Ostrand's Lemma~\ref{l:Ostrand}, we can additionally assume that  $\U=\bigcup_{i=1}^m\U_i$ where each family $\U_i$ consists of pairwise disjoint sets. By the normality of the compact space $Z$, every set $U\in\U$ contains a closed subset $F_U\subseteq U$ of $Z$ such that $Z=\bigcup_{U\in\U}F_U$. By Urysohn's Lemma \cite[1.5.11]{Eng}, there exists a continuous map $\xi_U:Z\to[0,1]$ such that $F_U\subseteq \xi_U^{-1}(1)$ and $Z\setminus U\subseteq \xi_U^{-1}(0)$. For every $U\in\U$ choose a point $z_U\in U$.  

Since the pair $(M,X)$ is everywhere $\vC$-preuniversal, for every $i\in\{1,\dots,m\}$ and $U\in\U_i$ there exists a map $\phi_U:K\to M$ such that $\phi_U[K]\subseteq B(w_i\circ g(z_U),\eta)$ and $\phi_U^{-1}[X]=C$. Consider the map $\varphi_U:F\to M$ defined by the formula
$$\varphi_U(x)=
\begin{cases}
\lambda\big(w_i\circ g\circ f(x),\phi_U(x),\xi_U\circ f(x)\big)&\mbox{if $f(x)\in U$};\\
w_i\circ g\circ f(x)&\mbox{otherwise}.
\end{cases}
$$

\begin{claim}\label{cl3} For every $i\in\{1,\dots,m\}$ and $U\in\U_i$ the map $\varphi_U:F\to M$ is continuous and has the following properties for every $x\in F$:
\begin{enumerate}
\item $d(\varphi_U(x),w_i\circ g\circ f(x))<2\delta$;
\item $f^{-1}[F_U]\cap \varphi_U^{-1}[X]=f^{-1}[F_U]\cap C$;
\item if $\xi_U\circ f(x)<1$, then $\varphi_U(x)\in X$.
\end{enumerate}
\end{claim}

\begin{proof} The continuity of $\varphi_U$ follows from the definition of $\varphi_U$ and the equality $\lambda(x,y,0)=x$ holding for all $(x,y)\in W_\lambda$. 
\smallskip

1. Fix any $x\in F$. If $f(x)\notin U$, then $\varphi_U(x)=w_i\circ g\circ f(x)\in B(w_i\circ g\circ f(x),2\delta)$. If $f(x)\in U$, then the choice of the cover $\U$ ensures that $$w_i\circ g\circ f(x)\in w_i\circ g[U]\subseteq B(w_i\circ g(z_U),\eta).$$ On the other hand, $\phi_U(x)\in B(w_i\circ g(z_U),\eta)$. Then $$(w_i\circ g\circ f(x),\phi_U(x))\in B(w_i\circ g(z_U),\eta)\times B(w_i\circ g(z_U),\eta)\subseteq W_\lambda$$ and $$\begin{aligned}
\varphi_U(x)&=\lambda(w_i\circ g\circ f(x),\phi_U(x),\xi_U\circ f(x))\in\\
&\in \lambda\big[B(w_i\circ g(z_U),\eta)\times B(w_i\circ g(z_U),\eta)\times[0,1]\big]\subseteq\\
&\subseteq B(w_i\circ g(z_U),\delta)\subseteq B(w_i\circ g\circ f(x),\delta+\eta)\subseteq B(w_i\circ g\circ f(x),2\delta).
\end{aligned}
$$
\smallskip

2. Fix any $x\in F$ with $f(x)\in F_U$. We need to prove that $\varphi_U(x)\in X$ if and only if $x\in C$. Since $f(x)\in F_U\subseteq \xi_U^{-1}(1)$, we have
$$\varphi_U(x)=\lambda(w_i\circ g\circ f(x),\phi_U(x),\xi_U\circ f(x))=\lambda(w_i\circ g\circ f(x),\phi_U(x),1)=\phi_U(x).$$
The choice of the map $\phi_U$ ensures that $\varphi_U(x)=\phi_U(x)\in X$ if and only if $x\in C$.
\smallskip

3. Assume that $\xi_U\circ f(x)<1$. Let us recall $g\circ f(x)\in g[Z]\subseteq X$. By the properties (i) and (ii), $$*_m(w_1\circ g\circ f(x),\dots,w_m\circ g\circ f(x))=g\circ f(x)\in X$$ and hence $w_j\circ g\circ f(x)\in X$ for all $j\in\{1,\dots,m\}$. In particular, $w_i\circ g\circ f(x)\in X$. If $\xi_U\circ f(x)=0$ or $\phi_U(x)=w_i\circ g\circ f(x)$, then 
$$\varphi_U(x)=\lambda(w_i\circ g\circ f(x),\phi_U(x),\xi_U\circ f(x))=w_i\circ g\circ f(x)\in X$$according to Lemma~\ref{l:lambda}(1,2). If $\xi_U\circ f(x)>0$ and $\phi_U(x)\ne w_i\circ g\circ f(x)$, then 
$$\varphi_U(x)=\lambda(w_i\circ g\circ f(x),\phi_U(x),\xi_U\circ f(x))\in X$$by Lemma~\ref{l:lambda}(3).
\end{proof}

For every $i\in\{1,\dots,m\}$, consider the map $\psi_i:F\to M$ defined by the formula
$$\psi_i(x)=\begin{cases}\varphi_U(x)&\mbox{if $x\in U$ for some $U\in\U_i$};\\
w_i\circ g\circ f(x)&\mbox{otherwise}.
\end{cases}
$$
It is easy to see that the map $\psi_i$ is well-defined and continuous.

Finally, consider the map $h:F\to M$ defined by the formula
$$h(x)=*_m(\psi_1(x),\dots,\psi_m(x))\quad\mbox{for \ $x\in K$}.
$$

We claim that the function $h$ has the desired properties.

\begin{claim} For every $x\in F$ we have $d(h(x),g\circ f(x))<\e$.
\end{claim}

\begin{proof} Claim~\ref{cl3} implies that $\psi_i(x)\in B(w_i\circ g\circ f(x),2\delta)$ for every $i\in\{1,\dots,m\}$ and hence
$$h(x)=*_m(\psi_1(x),\dots,\psi_m(x))\in B(g\circ f(x),\e)$$according to the choice of $\delta$.
\end{proof}

\begin{claim} $h^{-1}[X]=F\cap C$.
\end{claim}

\begin{proof} If $x\in F\cap C$, then the definition of $\psi_i$ and Claim~\ref{cl3} imply that $\psi_i(x)\in X$ and then $h(x)=*_m(\psi_1(x),\dots,\psi_m(x))\in X$ by the property (ii) of the operation $*_m$. Now assume that $x\in F\setminus C$. Since $Z=\bigcup_{U\in\U}F_U=\bigcup_{i=1}^m\bigcup_{U\in\U_i}F_U$, there exists $i\in\{1,\dots,m\}$ and $U\in\U_i$ such that $f(x)\in F_U$. By Claim~\ref{cl3}, $\psi_i(x)=\varphi_U(x)\notin X$. Now the property (ii) of the operation $*_m$ ensures that $h(x)=*_m(\psi_1(x),\dots,\psi_m(x))\notin X$.
\end{proof}
\end{proof}

A function $f:X\to Y$ between topological spaces is {\em proper} if for any compact set $K\subseteq Y$ the preimage $f^{-1}[K]$ is compact.

\begin{lemma}\label{l:proper} For any locally compact space $Z$, there exists a proper map $f:Z\to[0,\infty)$.
\end{lemma}

\begin{proof} The statement is trivial if $Z$ is compact. So we assume that $Z$ is not compact. Let $\alpha Z=Z\cup\{\infty\}$ be the one-point compactification of $Z$. Since $\alpha Z$ is metrizable and separable, there exists a continuous function $g:\alpha Z\to [0,1]$ such that $g^{-1}(0)=\{\infty\}$. It is easy to see that the function 
$$f:Z\to[1,\infty)\subseteq[0,\infty),\quad f:z\mapsto\frac1{g(z)},$$
is proper.
\end{proof}

\begin{lemma}\label{l10} Let $(K,C)\in\vC$, $F$ be a closed set in $K$ and $\e:K\setminus F\to(0,1]$ be a continuous function. Let $Z$ be a locally compact locally finite-dimensional space, $f:K\setminus F\to Z$ be a proper map, and $g:Z\to X$ be a continuous map. There exists a continuous map $h:K\setminus F\to M$ such that $h^{-1}[X]=C\setminus F$ and $d(h(x),g\circ f(x))<\e(x)$ for every $x\in K\setminus F$.
\end{lemma}

\begin{proof} By Lemma~\ref{l:proper}, there exists a proper map $p:Z\to [0,\infty)$. We lose no generality assuming that $0\in p[Z]$. For every $n\in\IN$, consider the compact subset $Z_n=p^{-1}\big[[0,n]\big]$ of $Z$ and the compact subset $K_n=f^{-1}[Z_n]$ of $K\setminus F$. Since the space $Z$ is locally finite-dimensional, the compact subspaces $Z_n$ of $Z$ are finite-dimensional.  

Let us recall that $x=u(x)*v(x)$ for any $x\in M$. In particular, for every $z\in Z$ we have $u(g(z))*v(g(z))=g(z)\in g[Z]\subseteq X$, which implies $u\circ g(z),v\circ g(z)\in X$.
  
For every $n\in\w$, by the compactness of the set $g[Z_{n+2}]\subseteq M$, there exists a number $\delta_n>0$ such that for any $x\in g[Z_{n+2}]$ and $x_1\in B(u(x),\delta_n)$, $x_2\in B(v(x),\delta_n)$ we have $d(x,x_1*x_2)<\min \e[Z_{n+2}]$. If $n>0$ then we can additionally assume that $\delta_n\le\delta_{n-1}$. By Lemma~\ref{l:lambda2}, there exists $\eta_n>0$ such that for every $x\in u\circ g[Z_{n+1}]\cup v\circ g[Z_{n+1}]$ we have $B(x,\eta_n)\times B(x,\eta_n)\subseteq W_\lambda$ and $$\lambda\big[B(x,\eta_n)\times B(x,\eta_n)\times[0,1]\big]\subseteq B(x,\delta_n).$$

For every $n\in\w$ consider the continuous piecewise linear function $\xi_n:[0,\infty)\to[0,1]$ such that 
$$\xi_n(t)=\begin{cases}
1&\mbox{if $n\le t\le n+1$};\\
1+3(t-n)&\mbox{if $n-\frac13\le t\le n$};\\
1-3(t-n-1)&\mbox{if $n+1\le t\le n+\frac43$};\\
0&\mbox{otherwise}.
\end{cases}
$$ 
By Lemma~\ref{l:step}, for every even number $n\in\w$, there exists a continuous map $\zeta_n:K_{n+1}\to M$ such that $\zeta_n^{-1}[X]=K_{n+1}\cap C$ and 
 $d(u\circ g\circ f(x),\zeta_n(x))<\eta_n$ for every $x\in K_{n+1}$. By the same Lemma~\ref{l:step}, for every odd number $n\in\w$ there exists a continuous map $\zeta_n:K_{n+1}\to M$ such that $\zeta_n^{-1}[X]=K_{n+1}\cap C$ and 
 $d(v\circ g\circ f(x),\zeta_n(x))<\eta_n$ for every $x\in K_{n+1}$.

Let $\varphi,\psi:K\setminus B\to M$ be maps defined by the formulas
$$\varphi(x)=\begin{cases}\lambda\big(u\circ g\circ f(x),\zeta_{2n}(x),\xi_{2n}(p\circ f(x))\big)&\mbox{if $p\circ f(x)\in [2n-\frac13,2n+\frac43]$ for some $n\in\w$};\\
u\circ g\circ f(x)&\mbox{otherwise};
\end{cases}
$$ 
and 
$$\psi(x)=\begin{cases}\lambda\big(v\circ g\circ f(x),\zeta_{2n+1}(x),\xi_{2n+1}(p\circ f(x))\big)&\mbox{if $p\circ f(x)\in [2n+\frac23,2n+\frac73]$ for some $n\in\w$};\\
v\circ g\circ f(x)&\mbox{otherwise}.
\end{cases}
$$ 
It can be shown that the maps $\varphi$ and $\psi$ are continuous. 

Finally, define a map $h:K\setminus B\to M$ by the formula
$$h(x)=\varphi(x)*\psi(x)\quad\mbox{for \ $x\in K\setminus B$}.$$
 
\begin{claim} $h^{-1}[X]=C\setminus F$.
\end{claim}

\begin{proof} Let $x\in K\setminus F$ be any point. Find a unique number $n\in\w$ such that $n\le p\circ f(x)<n+1$. We recall that $u\circ g\circ f(x),v\circ g\circ f(x)\in X$.

First we assume that $x\notin C$. In this case $\zeta_n(x)\notin X$ by the choice of $\zeta_n$.

If $n$ is even, then 
$$\varphi(x)=\lambda(u\circ g\circ f(x),\zeta_{n}(x),\xi_n\circ p\circ f(x))=\lambda(u\circ g\circ f(x),\zeta_n(x),1)=\zeta_n(x)\notin X$$ and hence
$h(x)=\varphi(x)*\psi(x)\in (M\setminus X)*M\subseteq M\setminus X$.

If $n$ is odd, then 
$$\psi(x)=\lambda(v\circ g\circ f(x),\zeta_{n}(x),\xi_n\circ p\circ f(x))=\lambda(v\circ g\circ f(x),\zeta_n(x),1)=\zeta_n(x)\notin X$$ and hence
$h(x)=\varphi(x)*\psi(x)\in M*(M\setminus X)\subseteq M\setminus X$.
In both cases we obtain that $h(x)\notin X$.

Now assume that $x\in C$. In this case $\zeta_k(x)\in X$ for any $k\in\w$ with $|k-n|\le 1$, which implies $\varphi(x)\in X$ and $\psi(x)\in X$, see Lemma~\ref{l:lambda}(3). Then $h(x)=\varphi(x)*\psi(x)\in X*X\subseteq X$.
\end{proof}

\begin{claim} $d(h(x),g\circ f(x))<\e(x)$ for all $x\in K\setminus F$.
\end{claim}

\begin{proof} Take any $x\in K\setminus F$ and find a unique number $n\in\w$ with $n\le p\circ f(x)<n+1$. Then $f(x)\in Z_{n+1}$ and $x\in K_{n+1}$. 

\begin{claim} $\varphi(x)\in B(u\circ g\circ f(x),\delta_{n-1})$.
\end{claim}

\begin{proof} The inclusion is trivially true whenever $\varphi(x)=u\circ g\circ f(x)$. So, we assume that $\varphi(x)\ne u\circ g\circ f(x)$. In this case $\varphi(x)=\lambda(u\circ g\circ f(x),\zeta_{2k}(x),\xi_{2k}\circ p\circ f(x))$ for some $k\in\w$ with $p\circ f(x)\in [2k-\frac13,2k+\frac43]$. It follows from $p\circ f(x)\in[n,n+1]\cap [2k-\frac13,2k+\frac43]$ that $|n-2k|\le 1$. The choice of $\zeta_{2k}$ ensures that $\zeta_{2k}(x)\in B(u\circ g\circ f(x),\eta_{2k})$ and the choice of $\eta_{2k}$ guarantees that $$\varphi(x)=\lambda(u\circ g\circ f(x),\zeta_{2k}(x),\xi_{2k}\circ p\circ f(x))\in B(u\circ g\circ f(x),\delta_{2k})\subseteq B(u\circ g\circ f(x)),\delta_{n-1}).$$
\end{proof} 

By analogy we can prove that $\psi(x)\in B(v\circ g\circ f(x),\delta_{n-1})$.

Now the choice of $\delta_{n-1}$ guarantees that
$$d(g\circ f(x),h(x))=d(g\circ f(x),\varphi(x)*\psi(x))<\min \e[Z_{n+1}]\le \e(x).$$
\end{proof}
\end{proof}

\begin{lemma}\label{l11} Let $K$ be a compact space, $F$ be a closed subspace of $K$ and $f:K\to M$ be a continuous map. For every continuous function $\e:K\setminus F\to(0,1]$, there exist a locally compact locally finite-dimensional space $Z$, a proper map $g:K\setminus F\to Z$, and a map $h:Z\to X$ such that $d(f(x),h\circ g(x))\le\e(x)$ for every $x\in K\setminus F$.
\end{lemma}

\begin{proof}  We can embed $K$ into the Hilbert cube $[0,1]^\w$ so that $K$ is disjoint with the set $\bigcup_{n\in\w}([0,1]^n\times[0,1]^{\w\setminus n})$. Since $M$ is an ANR, the map $f:K\to M$ can be extended to a map $\bar f:O_K\to M$ defined on some neighborhood $O_K$ of $K$ in $[0,1]^\w$. We lose no generality assuming that $O_K$ is of basic form $O_K=W\times [0,1]^{\w\setminus m}$ for some $m\in\w$ and some open set $W$ in $[0,1]^m$.

For every $n\in\w$ consider the projection $\pi_n:[0,1]^\w\to [0,1]^n$, $\pi_n:x\mapsto x{\restriction}n$, to the first $n$ coordinates.

Choose a decreasing sequence of open sets $(U_n)_{n\in\w}$ in $K$ such that
\begin{itemize}
\item $U_0=U_1=K$;
\item for every $n\in\IN$ the closure of $U_n$ is contained in $U_{n-1}$ and
\item each open neighborhood $O_F$ of the set $F$ in $K$ contains some set $U_n$.
\end{itemize}

Using the continuity of the map $\bar f$ and the compactness of the space $K$, we can construct an increasing sequence of numbers $(m_n)_{n\in\w}$ such that $m_0=0$, $m_1\ge m$ and for every $n\in\IN$, $x\in K$ and $y\in O_K$ with $\pi_{m_n}(x)=\pi_{m_n}(y)$ we have $$d(f(x),\bar f(y))<\tfrac12\min \e[K\setminus U_{n+1}].$$ 

For every $i\in\w$ find a unique number $n\in\w$ with $m_n\le i<m_{n+1}$ and choose a continuous function $\xi_i:K\to[0,1]$ such that $\xi_i[K\setminus U_n]\subseteq\{0\}$ and $\xi_i[U_{n+1}]\subseteq\{1\}$. It follows that $\xi_i[K]=\xi_i[U_1]\subseteq\{1\}$ for every $i<m_1$. Consider the continuous map 
$$\bar g:K\to[0,1]^\w,\quad \bar g:(x_i)_{i\in\w}\mapsto (\xi_i(x)\cdot x_i)_{i\in\w}.$$
Taking into account that $m\le m_1$ and $\xi_i[K]\subseteq\{1\}$ for every $i<m_1$, we conclude that $\pi_m=\pi_m\circ \bar g$ and hence the map $\bar f\circ\bar g:K\to M$ is well-defined.

\begin{claim}\label{cl9} For every $n\in\w$ and $x\in U_n\setminus U_{n+1}$ we have 
$$d(f(x),\bar f\circ \bar g(x))\le \tfrac12\e(x).$$
\end{claim}

\begin{proof} For every $i<m_n$ we can find a unique $k<n$ such that $m_k\le i<m_{k+1}\le m_n$ and conclude that $\xi_i[U_n]\subseteq\xi_i[U_{k+1}]\subseteq \{1\}$. Then $\pi_{m_n}(\bar g(x))=\pi_{m_n}(x)$ and hence $$d(f(x),\bar f\circ\bar g(x))<\tfrac12\min\e[K\setminus U_{n+1}]\le\tfrac12\e(x).$$
\end{proof}

\begin{claim}\label{cl10} For every $n\in\w$ and $x\in U_n\setminus U_{n+1}$ we have 
$$\bar g(x)\in [0,1]^{m_{n+1}}\times\{0\}^{\w\setminus m_{n+1}}\subseteq [0,1]^\w\setminus K.$$
\end{claim}

\begin{proof} Given any $i\ge m_{n+1}$, find a unique $k\ge n+1$ such that $m_k\le i<m_{k+1}$ and conclude that $\xi_i(x)\in \xi_i[K\setminus U_{n+1}]\subseteq \xi_i[K\setminus U_k]\subseteq\{0\}$. This implies that $\bar g(x) \in [0,1]^{m_{n+1}}\times\{0\}^{\w\setminus m_{n+1}}\subseteq[0,1]^\w\setminus K$.
\end{proof}

Claim~\ref{cl10} implies that the space $$Z=\bar g[K\setminus F]=\bigcup_{n\in\w}\bar g[U_n\setminus U_{n+1}]=\bar g[K]\setminus \bar g[F]$$ is locally compact and locally finite-dimensional. Consider the map $g=\bar g{\restriction}K\setminus F:K\setminus F\to Z$ and observe that $g$ is proper.

\begin{claim} There exists a continuous map $\delta:Z\to(0,1]$ such that $\delta\circ g(x)\le\frac12 \e(x)$ for every $x\in K\setminus F$.
\end{claim}

\begin{proof} Let $\U$ be a locally finite cover of the locally compact space $Z$ by open subsets with compact closures. Since the map $g:K\setminus F\to Z$ is proper, for every $U\in\U$ the set $g^{-1}[\overline{U}]\subseteq K\setminus F$ is compact and hence the number $\e_U=\frac12{\cdot}\min \e\big[g^{-1}[\overline{U}]\big]$ is strictly positive. 

By the paracompactness of $Z$, there exists a family of continuous maps $(\lambda_U:Z\to[0,1])_{U\in\U}$ such that
\begin{itemize}
\item $\lambda^{-1}\big[(0,1]\big]\subseteq U$ for every $U\in\U$ and 
\item $\sum_{U\in\U}\lambda_U(x)=1$ for every $x\in K\setminus F$.
\end{itemize}
Then the function $$\delta:Z\to[0,1],\quad \delta:x\mapsto\sum_{U\in\U}\lambda_U(x)\e_U,$$has the required property.
\end{proof}

Finally, consider the continuous function $$h:Z\to X,\quad h:z\mapsto H(\bar f(z),\delta(z)),$$
and observe that for every $x\in K\setminus F$ we have
$$
\begin{aligned}
d(f(x),h\circ g(x))&=d(f(x),H(\bar f\circ \bar g(x),\delta\circ g(z))\le\\
&\le d\big(f(x),\bar f\circ \bar g(x)\big)+d\big(\bar f\circ g(x),H(\bar f\circ g(x),\delta\circ g(x))\big)\le\\
&\le\tfrac12\e(x)+\delta(x)\le \tfrac12\e(x)+\tfrac12\e(x)=\e(x).
\end{aligned}
$$

\end{proof}

Finally, we are able to finish the proof of Theorem~\ref{t:main}.

\begin{lemma} The pair $(M,X)$ is strongly $\vC$-preuniversal.
\end{lemma}

\begin{proof} Take any open cover $\U$ of $M$, pair $(K,C)\in\vC$, closed subset $B\subseteq K$ and map $f:K\to M$ such that $(f{\restriction}B)^{-1}[X]=B\cap C$. By Lemma~\ref{l1}, there exists $\e_0>0$ such that for any $x\in f[K]$ there exists $U\in\U$ such that $B(f(x),2\e_0)\subseteq U$. Let $\e:K\to [0,\e_0)$ be a continuous function such that $\e^{-1}(0)=B$. By Lemma~\ref{l11}, there exist a locally compact locally finite-dimensional space $Z$, a proper map $p:K\setminus B\to Z$ and a map $g:Z\to X$ such that $d(f(x),g\circ p(x))\le\e(x)$ for every $x\in K\setminus B$. 

By Lemma~\ref{l10}, there exists a continuous map $h:K\setminus B\to M$ such that $h^{-1}[X]=C\setminus B$ and $d(h(x),g\circ p(x))<\e(x)$ for every $x\in K\setminus B$. The triangle inequality implies that for every $x\in K\setminus B$ we have
$$d(h(x),f(x))\le d(h(x),g\circ p(x))+d(g\circ p(x),f(x))<2\e(x)\le 2\e_0,$$
which implies that the map $h$ is $\U$-near to $f{\restriction}K\setminus B$.

Consider the map $\bar f:K\to M$ such that $\bar f{\restriction}B=f{\restriction}B$ and $\bar f{\restriction}K\setminus B=h$. The choice of the function $\e$ with $\e^{-1}(0)=B$ and the continuity of the function $f$ imply the continuity of the function $\bar f$. Finally, observe that 
$$\bar f^{-1}[X]=(f{\restriction}B)^{-1}[X]\cup (f{\restriction}K\setminus B)^{-1}[X])=(B\cap C)\cup h^{-1}[X]=(B\cap C)\cup (C\setminus B)=C.$$
\end{proof}



\begin{thebibliography}{}


\bibitem{BGM} J.~Baars, H.~Gladdines, J.~van Mill, {\em Absorbing systems in infinite-dimensional manifolds}, Topology Appl. {\bf 50}:2 (1993), 147--182.

\bibitem{Ban98} T.~Banakh, {\em Characterization of spaces admitting a homotopy dense embedding into a Hilbert manifold}, Topology Appl. {\bf 86}:2 (1998),  123--131.

\bibitem{Ban} T.~Banakh, {\em The strongly universal property in closed convex sets}, Mat. Stud. {\bf 10}:2 (1998), 203--218.

\bibitem{BC} T.~Banakh, R.~Cauty, {\em Interplay between strongly universal spaces and pairs}, Dissert. Math. 386 (2000), 38 pp.

\bibitem{BDP} T.~Banakh, T.~Dobrowolski, A.~Plichko, {\em  Applications of some results of infinite-dimensional topology to the topological classification of operator images and weak unit balls of Banach spaces}, Dissert. Math. {\bf 387} (2000), 1--81.

\bibitem{BKO} T.~Banakh, P.~Krupski, K.~Omiljanowski, {\em Hyperspaces of infinite compacta with finitely many accumulation points}, in preparation.

\bibitem{BanMaz} T.~Banakh, N.~Mazurenko, {\em The topology of systems of hyperspaces determined by dimension functions}, Topology, {\bf 48} (2009) 43--53.

\bibitem{BRZ} T.~Banakh, T.~Radul, M.~Zarichnyi, {\em Absorbing sets in infinite-dimensional manifolds}, VNTL Publ., Lviv, 1996.

\bibitem{BM} M.~Bestvina, J.~Mogilski, {\em Characterizing certain incomplete infinite-dimensional absolute retracts}, Michigan Math. J. {\bf 33}:3 (1986) 291--313.

\bibitem{Cauty} R.~Cauty, {\em Suites $F_\sigma$-absorbantes en th\'eorie de la dimension}, Fund. Math. {\bf 159}:2 (1999), 115--126.

\bibitem{Eng} R.~Engelking, {\em General Topology}, Heldermann Verlag, Berlin, 1989. 

\bibitem{End} R.~Engelking, {\em Theory of dimensions finite and infinite}, Heldermann Verlag, Lemgo, 1995. 

\bibitem{BGM} J.~Baars, H.~Gladdines, J.~van Mill, {\em Absorbing systems in infinite-dimensional manifolds}, Topology Appl. {\bf 50}:2 (1993), 147--182.

\bibitem{DR} T.~Dobrowolski, L.~Rubin, {\em The hyperspaces of infinite-dimensional compacta for covering and cohomological dimension are homeomorphic}, Pacific J. Math. {\bf 164}:1 (1994), 15--39.

\bibitem{GM} H.~Gladdines, J.~van Mill, {\em Hyperspaces of infinite-dimensional compacta}, Compositio Math. {\bf 88}:2 (1993), 143--153.

\bibitem{Krup} P.~Krupski, {\em Hyperspaces of continua with connected boundaries in $\pi$-Euclidean Peano continua}, Topology Appl. {\bf 270} (2020), 106954, 6 pp.

\bibitem{KS07} P.~Krupski, A.~Samulewicz, {\em Strongly countable dimensional compacta form the Hurewicz set}, Topology Appl. {\bf 154}:5 (2007),  996--1001.

\bibitem{KS} P.~Krupski, A.~Samulewicz, {\em More absorbers in hyperspaces}, Topology Appl. {\bf 221} (2017), 352--369.

\bibitem{KK} K.~Kr\'olicki, P.~Krupski, {\em Wilder continua and their subfamilies as coanalytic absorbers}, Topology Appl. {\bf 220} (2017), 146--151.


\bibitem{KSY} W.~Kubi\'s, K.~Sakai, M.~Yaguchi, {\em Hyperspaces of separable Banach spaces with the Wijsman topology}, Topology Appl. {\bf 148}:1-3 (2005) 7--32.

\bibitem{Ost} P.~Ostrand, {\em Dimension of metric spaces and Hilbert's problem $13$}, Bull. Amer. Math. Soc. {\bf 71} (1965), 619--622.

\bibitem{Sak1} K.~Sakai, {\em Geometric aspects of general topology}, Springer,  Tokyo, 2013.

\bibitem{Sak2} K.~Sakai, {\em Topology of infinite-dimensional manifolds}, Springer, Tokyo, 2020.

\bibitem{Tor80} H.~Toru\'nczyk, {\em On ${\rm CE}$-images of the Hilbert cube and characterization of $Q$-manifolds}, Fund. Math. {\bf 106}:1 (1980), 31--40.

\bibitem{Tor81} H.~Toru\'nczyk, {\em Characterizing Hilbert space topology}, Fund. Math. {\bf 111}:3 (1981), 247--262.

\end{thebibliography}
\end{document}